\documentclass[12pt]{amsart}
\usepackage{amssymb, eucal, amsfonts, amsmath, xy,latexsym}
\def\k{{\Bbbk}}

\def\H{{\mathcal H}}

\def\OO{{\mathcal O}}
\def\bC{{\mathbb C}}

\def\g{{\mathfrak g}}

\def\fl{{\mathfrak l}}

\def\NN{{\mathcal N}}

\def\V{{\mathcal V}}

\def\Z{{\mathbb Z}}

\def\End{\mathop{\fam0 End}}

\def\Lie{\mathop{\fam0 Lie}}
\def\Ad{\mathrm{Ad}\,}

\def\ad{\mathrm{ad\,}}

\def\la{\langle}
\def\ra{\rangle}

\theoremstyle{plain}
\newtheorem{theorem}{Theorem}[section]
\newtheorem{corollary}{Corollary}[section]
\newtheorem{prop}{Proposition}[section]
\newtheorem{lemma}{Lemma}[section]

\theoremstyle{definition}

\theoremstyle{remark}

\newtheorem{rem}{Remark}[section]

\def\subtitle#1. {{\medskip\bf#1\par\nobreak\smallskip}}
\def\proclaim#1. {\medbreak\bgroup\noindent\bf#1. \it}

\def\endproclaim{\egroup
\ifdim\lastskip<\medskipamount\removelastskip\medskip\fi}
\newcount
=0
\def\citedef#1 {\advance\citation by1
  \expandafter\edef\csname#1\endcsname{{\the\citation}}
  \checkendcitedef}
\def\checkendcitedef#1{\ifx#1\endcitedef\else\citedef#1\fi}
\def\cite#1{\csname#1\endcsname}
\citedef  Bor Bou Ca CP  Hes78  Hes78a   HSp
 LT LS  L2 L3 P03 PS Ros Stein  St  Vog Xue1
\endcitedef
\newtoks\nextauth
\newif\iffirstauth
\def\checkendauth#1{\ifx\endauth#1
        \iffirstauth\the\nextauth
        \else{} and \the\nextauth\fi,
    \else\iffirstauth\the\nextauth\firstauthfalse
        \else, \the\nextauth\fi
        \expandafter\auth\expandafter#1\fi}
\def\auth#1 #2 {\nextauth={#1 #2}\checkendauth}
\newif\ifinbook
\newif\ifbookref
\def\nextref#1 {\bookreffalse\inbookfalse
    \bibitem[\cite{#1}]{}
    \firstauthtrue
    \ignorespaces}
\def\paper#1{{\it#1,}}

\def\book#1{\bookreftrue{\it#1,}}
\def\journal#1{#1\ifinbook,\fi}
\def\bookseries#1{#1,}
\def\Vol#1{\ifbookref Vol. #1,\else\ifinbook Vol. #1,\else{\bf#1}\fi\fi
    \space\ignorespaces}

\def\publisher#1{#1,}
\def\Year#1{\ifbookref #1.\else\ifinbook #1,\else(#1)\fi\fi
    \space\ignorespaces}
\def\Pages#1{\ifinbook pp. #1.\else #1.\fi}
\begin{document}
\title{Hesselink strata in small characteristic and Lusztig--Xue pieces}
\author{Alexander Premet}
\thanks{\nonumber{\it Mathematics Subject Classification 21020}.
Primary 14L30, 17B45. Secondary 14L15, 20G41. {\it \mbox{\,\quad}Key
words and phrases}: simple algebraic groups, nilpotent orbits, Hesselink strata}
\address{Department of Mathematics, The University of Manchester, Oxford Road, Alan Turing Building, 
M13 9PL, UK} \email{alexander.premet@manchester.ac.uk}
\begin{abstract}
\noindent Let $G$ be a connected reductive algebraic group over an algebraically closed field $\k$
of characteristic $p\ge 0$ and $\g=\Lie(G)$.
In this paper we show that the nilpotent pieces ${\rm LX}(\Delta)$ introduced by Lusztig form a partition of  the nilpotent cone of $\g$ and hence coincide with the 
Hesselink strata $\H(\Delta)$ where $\Delta$ runs through the set of all weighted Dynkin diagrams of $G$. Thanks to earlier results obtained by Lusztig, Xue and Voggesberger this boils down to describing the pieces ${\rm LX}(\Delta)$ for groups of type ${\rm E_7}$ in characteristic $2$ and for groups of type ${\rm E_8}$ in characteristic $2$ and $3$. 
Our arguments are computer-free, but rely very heavily on the results of 
Liebeck--Seitz obtained in [\cite{LS}].
\end{abstract}

\maketitle

\medskip

\medskip 

\begin{center}
{\it In memory of Gary Seitz}
\end{center}

\bigskip

\medskip
\section{Introduction}
Let $G$ be a connected reductive algebraic group of rank $\ell$ over an algebraically closed field $\k$ and $T$ a maximal torus of $G$. Let $\Sigma$ be the root system of $G$ with respect to $T$ and  $\Pi$ a basis of simple roots of $\Sigma$.
Write $X(T)$ (resp. $X_*(T)$) for the lattice of rational characters
(resp. cocharacters) of $T$
 and $X_*^+(T)$ for  the intersection of $X_*(T)$ with the dual Weyl chamber of $X_*(T)\otimes_\Z\mathbb{R}$ associated with $\Pi$. 
Each rational cocharacter $\lambda\in X_*(G)$ gives rise to a $\Z$-grading $$\g=\textstyle{\bigoplus}_{i\in\Z}\,\g(\lambda,i),\ \ \ \ \g(\lambda,i)\,=\,\{x\in\g\,|\,\,(\Ad \lambda(t))\,x=t^ix\ \ \mbox{for all}\ \  t\in\k^\times\}$$  of the Lie algebra $\g=\Lie(G)$. For $d\in\Z$, we put $\g(\lambda,\ge d):=\bigoplus_{i\ge d}\,\g(\lambda, i)$ and   $\g(\lambda,< d):=\bigoplus_{i< d}\,\g(\lambda, i)$, and denote by
$P(\lambda)=L(\lambda)R_u(\lambda)$ the parabolic subgroup of $G$ associated with $\lambda$. 
Here $L(\lambda)=Z_G(\lambda)$ is a Levi subgroup of $G$.
Recall that $\Lie(P(\lambda))=\g(\lambda,\ge 0)$ and $\Lie(L(\lambda))=\g(\lambda,0)$.

We let $\NN(\g)$ denote the nilpotent cone of $\g$, the variety of all $(\Ad G)$-unstable vectors of $\g$, and write $\mathfrak{D}_G$ for the set of all 
{\it Dynkin labels}
 attached to the nilpotent orbits of a complex Lie algebra with root system $\Sigma$. 
As explained in [\cite{CP}], the Hesselink strata of $\NN(\g)$ are parameterised by the set of cocharacters $\tau_\Delta\in X_*^+(T)$ with $\Delta\in\mathfrak{D}_G$ and form a partition of $\NN(\g)$, so that
$$\NN(\g)\,=\,\textstyle{\bigsqcup}_{\Delta \in \mathfrak{D}_G}\,\H(\Delta).$$ 
The cocharacter $\tau_\Delta$ can be read off the weighted Dynkin diagram $(a_1,\ldots,a_\ell)$ associated with the complex nilpotent orbit with label $\Delta$
as follows:
if $x$ is a root vector of $\g$ corresponding to a simple root $\alpha_i\in\Pi$  then 
$(\Ad \tau_\Delta(t)))(x)=t^{a_i}x$ for all $t\in\k^\times$. The Hesselink stratum attached to $\tau_\Delta$ has the form
$$\H(\Delta)\,=\,(\Ad G)\,\big(\V(\tau_\Delta,2)_{ss}+\g(\tau_\Delta,\ge 3)\big)$$ where $\V(\tau_\Delta,2)_{ss}\ne \varnothing$ is the set of all $(\Ad L(\tau_\Delta))$-semistable vectors of $\g(\tau_\Delta,2)$ (see [\cite{CP}] for more detail).

Given $\Delta\in \mathfrak{D}_{G}$  we write $\g_{2}^{\Delta,!}$ for the set of all $x\in\g(\tau_\Delta,2)$
such that $G_{x}\subset P(\tau_\Delta)$ where $G_{x}=Z_G(x)$ is the stabiliser of $x$ in $G$. As explained in [\cite{CP}, Remark~7.3] each set $\g_{2}^{\Delta,!}$  contains $\V(\tau_\Delta,2)_{ss}$, a nonempty Zariski open subset of $\g(\tau_\Delta,2)$.  The set $${\rm LX}(\Delta)\,:=\,(\Ad G)\big(\g_{2}^{\Delta, !}+\g(\tau_\Delta,\ge 3)\big)$$ containing $\H(\Delta)$ will be referred to as
 the {\it Lusztig-Xue piece} of $\NN(\g)$ associated with $\Delta$.
The pieces ${\rm LX}(\Delta)$ and their analogues for $\NN(\g^*)$ and for the unipotent variety of $G$ were introduced by Lusztig.
Viability of these pieces has to do with the fact that $\g_2^{\Delta,!}$ is defined in a more transparent fashion than its elusive subset $\V(\tau_\Delta,2)_{ss}$.

In [\cite{L2}, Appendix~A], Lusztig and Xue proved that the pieces ${\rm LX}(\Delta)$ form a  partition of $\NN(\g)$ 
in the case where $G$ is a classical group. 
Very recently, the same property was established by Voggesberger for groups of type ${\rm G}_2$, ${\rm F}_4$ and ${\rm E}_6$ (see our discussion in \ref{4.1} for more detail). These results imply that  $\rm{LX}(\Delta)=\H(\Delta)$ for all $\Delta\in\mathfrak{D}_{G}$ provided that $G$ is not of type ${\rm E_7}$ or ${\rm E_8}$. 

The partition property of the coadjoint analogues of ${\rm LX}(\Delta)$ was established by Lusztig [\cite{L3}] and Xue [\cite{Xue1}] in all cases where $G$ is a simple algebraic group and $p={\rm char}(\k)$ equals the ratio of the squared lengths of long and short roots in $\Sigma$. In all other cases
there is a $G$-equivariant bijection between $\NN(\g)$ and $\NN({\g}^*)$ which enables one to identify the nilpotent coadjoint orbits and pieces of $\g^*$ with those of $\g$; see [\cite{PS}, Section~5.6].

 It was conjectured in [\cite{Vog}] that the equality 
 $\rm{LX}(\Delta)=\H(\Delta)$ 
should also hold for all $\Delta\in\mathfrak{D}_G$ in the case where $G$ is a group of type ${\rm E_7}$ or ${\rm E_8}$. 
Our goal in this paper is to prove this conjecture. 

The orbits $\OO(e)=(\Ad G)\,e$ with $e\in\NN(\g)$
will be denoted by their Dynkin labels $\Delta$ or their
variants $(\Delta)_p$. The latter are attached to a small number of new nilpotent orbits which appear when $(\Sigma,p)\in\{({\rm G_2}, 3), ({\rm F_4}, 2), ({\rm E_7}, 2), ({\rm E_8}, 2), ({\rm E_8}, 3)\}$; see [\cite{LS}] for detail.
Combining our results obtained in Section~2 with the results of Lusztig, Xue and Voggesberger mentioned above we obtain the following:
\begin{theorem}\label{main}
	Let $G$ be a connected reductive group over an algebraically closed field $\k$ of characteristic $p\ge 0$. Then  $\H(\Delta)={\rm LX}(\Delta)$ for all $\Delta\in\mathfrak{D}_G$ and hence $$\NN(\g)\,=\,\textstyle{\bigsqcup}_{\Delta \in \mathfrak{D}_G}\,{\rm LX}(\Delta).$$ 
\end{theorem}
 Since proving Theorem~\ref{main} reduces quickly to the case where $\Sigma$ is an irreducible root system, we may assume without loss of generality that our algebraic group $G$ is simple and simply connected.  
 
 We use Steinberg's notation $x_\alpha(t)$ for elements of the unipotent root subgroups $U_{\alpha}$ of $G$; see [\cite{Stein}, \S 3]. Simple root vectors
$e_{\alpha_i}$ with $\alpha_i\in\Pi$ are denoted by $e_i$, and we always use Bourbaki's numbering [\cite{Bou}] of simple roots. We assume that root vectors $e_\gamma\in\g_\gamma$ come from a Chevalley basis of an admissible lattice $\g_\Z\subset \g_\bC$, where $\g_\bC$ is a complex Lie algebra with root system $\Sigma$. 
Since we mostly work over fields of characteristic $2$, the signs of structure constants  do not really affect our computations. 

The Weyl group of $\Sigma$ is denoted by $W$ and we write $\la\,\cdot\, ,\cdot\,\ra$ for the canonical pairing between $X(T)$ and $X_*(T)$ with values in $\Z$. We fix a $W$-invariant $\mathbb{Q}$-valued inner product $(\,\cdot\,|\,\cdot\,)$ on $X_*(T)_\mathbb{Q}=X_*(T)\otimes_\Z\mathbb{Q}$ which enables one to identify $X_*(T)_\mathbb{Q}$ with the dual vector space $X(T)_\mathbb{Q}=X(T)\otimes_\Z\mathbb{Q}$.

When we need to specify a particular root vector, we sometimes follow the conventions of [\cite{LS}]. For example,  a root vector $e_\gamma$ with $\gamma=\alpha_2+\alpha_3+2\alpha_4+2\alpha_5+\alpha_6$ is denoted by $e_{234^25^26}$.
When confusion is unlikely, we prefer standard conventions for specifying root vectors.

We occasionally use the $(\Ad\,G)$-invariant $\Z$-valued  bilinear form $\kappa$ on a minimal admissible lattice $\g_\Z\subset \g_\mathbb{C}$ introduced in [\cite{CP}, 7.2]. Its reduction modulo $p$ will be denoted by the same symbol.
When describing certain cocharacters $\tau\in X_*(T)$ we often specify their effect
on the root vectors $e_i$ where $1\le i\le \ell$. More precisely, if $(\Ad \tau(t))\,e_i=t^{r_i}\, e_i$ for all $t\in\k^\times$, then we write $\tau=(r_1,\ldots, r_\ell)$. This will cause no confusion since $\Pi=\{\alpha_1,\ldots,\alpha_\ell\}$ is a $\mathbb{Q}$-basis of $X(T)_\mathbb{Q}$. 

If ${\rm char}(\k)=p>0$ then $\g=\Lie(G)$ carries a canonical restricted Lie algebra structure $\g\ni x\mapsto x^{[p]}\in \g$ equivariant under the adjoint action of $G$. 
It is well known that the nilpotent cone $\NN(\g)$ coincides with the set of all $x\in\g$ such that $x^{[p]^N}=0$ fore $N\gg 0$.  
A restricted Lie subalgebra $\mathfrak{a}$ of $\g$ is called $[p]$-{\it nilpotent} (resp. {\it toral}) if $\mathfrak{a}\subseteq\NN(\g)$ (resp. if the $[p]$-mapping $x\mapsto x^{[p]}$ is one-to-one on $\mathfrak{a}$). Given $x\in\g$ we write $\g_x$ for the centraliser of $x$ in the Lie algebra $\g$, and we often use the fact that $\Lie(G_x)\subseteq \g_x$.
If $x\in\g(\lambda,r)$ for some $\lambda\in X_*(G)$ and $r\in\Z$ then $\g_x=\bigoplus_{i\in\Z}\,\g_e(\lambda,i)$ where $\g_e(\lambda,i)=\g_e\cap \g(\lambda,i)$.
 We say that $x\in\g$ is  {\it toral} if $x^{[p]}=x$
\section{Hesselink strata and Lusztig--Xue pieces}
\subsection{}\label{4.1}
Let $\tau=\tau_\Delta\in X_*(T)$ be the cocharacter associated with $\Delta\in\mathfrak{D}_G$ and write $\g_{2}^{\Delta,!}$ for the set of all $x\in\g(\tau,2)$
such that $G_{x}\subset P(\tau)$. 
Since the parabolic subgroup $P(\tau)$ is optimal for all $x\in\V(\tau,2)_{ss}$ in the sense of the Kempf--Rousseau theory, we have the inclusion
$\V(\tau,2)_{ss}\subseteq\g_{2}^{\Delta,!}$.  From [\cite{CP}, Theorem~6.1(iii)] it follows that $\V(\tau,2)_{ss}$ is a nonempty Zariski open subset of $\g(\tau,2)$.  
The set $${\rm LX}(\Delta)\,:=\,(\Ad G)\big(\g_{2}^{\Delta, !}+\g(\tau,\ge 3)\big)$$ is called the {\it Lusztig-Xue piece} of $\NN(\g)$ associated with $\Delta$.
By the above,
${\rm LX}(\Delta)$ contains  $\H(\Delta)$ for every $\Delta\in\mathfrak{D}_G$.

The pieces ${\rm LX}(\Delta)$ were first introduced by Lusztig [\cite{L2}] in the Lie algebra case. In [\cite{L3}], the definition was extended to cover the coadjoint nullcone  $\NN(\g^*)$ and the unipotent variety $\mathfrak{U}(G)$ of $G$. 
In [\cite{L2}, Appendix~A], Lusztig and Xue proved that the following decomposition
\begin{equation}\label{E1} \NN(\g)\,=\,\textstyle{\bigsqcup}_{\Delta\in \mathfrak{D}_{G}}\, \mathrm{LX}(\Delta)\end{equation}   
holds for all groups of type $\rm A$, $\rm B$, $\rm C$, $\rm D$, the key point being that the union is disjoint. Very recently, the same property was established by L.~Voggesberger for groups of type ${\rm G}_2$, ${\rm F}_4$ and ${\rm E}_6$ 
with the help of {\sc Magma}; see [\cite{Vog}, Theorem~1.1]. 

Note that if (\ref{E1}) holds for $G$ then $\rm{LX}(\Delta)=\H(\Delta)$ for all $\Delta\in\mathfrak{D}_{G}$; see [\cite{CP}, Remark~7.3.1] for detail. It was conjectured in [\cite{Vog}, Conjecture~1.2] that (\ref{E1})  should also hold for the $\k$-groups of type ${\rm E_7}$ and ${\rm E_8}$. A similar expectation (covering $\NN(\g)$, $\NN(\g^*)$ and $\mathfrak{U}(G)$) was expressed by Lusztig in [\cite{L3}, 2.3].

 In this paper we aim to confirm Voggesberger's conjecture. Lusztig's expectation related to partitioning the unipotent variety of $G$ will be discussed in \ref{4.13}.
For completeness, we also provide a new proof for groups of type ${\rm G_2}$, ${\rm F_4}$ and ${\rm E_6}$.

  Our task will become much simpler if we establish the following: 
 \begin{equation}\label{S1}
 \mbox{\it If $e\in \g(\tau_\Delta,2)$ and $G_e\subset P(\tau_\Delta)$ then $e\in \H(\Delta)$}.
 \end{equation}

 Proving statement~(\ref{S1}) will occupy the main body of the paper.
From now on we assume that the group $G$ is simple, simply connected, and has type ${\rm E}_6$, ${\rm E}_7$, ${\rm E}_8$, ${\rm F}_4$ or ${\rm G}_2$. Let $\kappa$ denote the normalised Killing form introduced in [\cite{CP}, 7.3]. If all roots of $\Sigma$ have the same length then the radical of $\kappa$ coincides with the central toral subalgebra of $\g$; see [\cite{CP}, Lemma~7.3]. 
 We start with a reduction lemma which will also explain why (\ref{E1}) always holds in good characteristic.
\begin{lemma}\label{L1}
Suppose $G$ is a group of type $\rm E$ and let $\tau=\tau_\Delta$, where $\Delta\in \mathfrak{D}_{G}$.
If $e\in\g(\tau,2)$ is such that $G_{e}\subset P(\tau)$ and
$\g_{e}=\Lie(G_{e})$, then $e\in\H(\Delta)$. 
\end{lemma}
\begin{proof} 
As $G$ is simply connected, the centre of the Lie algebra $\g$ is spanned by a toral element $z\in\g(\tau,0)$ which is nonzero if and only if $(\Sigma,p)$ is one of $({\rm E_6},3)$ of $({\rm E_7},2)$. For $i\ge 0$, we let $[e,\g(\tau,i)]^\perp$ denote the set of all $x\in \g(\tau,-i-2)$ such that $\kappa(x,[e,\g(\tau,i)])=0$. Since $\kappa$ is $(\Ad G)$-invariant
we have that $$[e,\g(\tau,i)]^\perp=\{x\in \g(-i-2)\,|\,\,[e,x]\in {\rm Rad}\,\kappa\}.$$ As
${\rm Rad}\,\kappa=\k z\subset \g(\tau,0),$ we get 
$[e,\g(\tau,i)]^\perp=\g_e(\tau,-i-2)$ for $i\ge 1$ and $[e,\g(\tau,0)]^\perp=\{x\in \g(\tau,-2)\,|\,\,[e,x]\in \k z\}$.

By our assumption, $\g_{e}=\Lie(G_{e})\subset \Lie(P(\tau))=\bigoplus_{\i\ge 0}\,\g(\tau,i).$ As $\g_e=\bigoplus_{i\in\Z}\,\g_e(\tau,i)$ it follows that $\g_e(\tau,i)=0$ for $i\le -1$, forcing $\g(\tau,i+2)=[e,\g(\tau,i)]$ for all $i\ge 1$. Also, $\g(\tau,2)=[e,\g(\tau,0)]$ when $z=0$. If $z\ne 0$ we can only say at this point that the subspace $[e,\g(\tau,0)]$ has codimension $\le 1$ in $\g(\tau,2)$.

Since $\g_e=\Lie(G_e)$ we have that \begin{eqnarray*}\dim (\Ad G)\,e&=&\dim\g-\dim\g_e= \textstyle{\sum}_{i\in\Z}\,\dim\g(\tau,i)-\textstyle{\sum}_{i\ge 0}\,\dim \g_e(\tau,i)\\
	&=&\textstyle{\sum}_{i<0}\,\dim\g(\tau,i)+\textstyle{\sum}_{i\ge 0}\,\dim\, [e,\g(\tau,i)]\\
	&=&\dim \,[e,\g(\tau,0)]+\textstyle{\sum}_{i\not\in \{0,1,2\}}\,\dim \g(\tau,i).
	\end{eqnarray*}
Since it follows from [\cite{HSp}, Theorem~2] that $\dim \,(\Ad G)\,e$ 
and $\textstyle{\sum}_{i\not\in \{0,1\}}\,\dim \g(\tau,i)$ are even numbers, it must be that
$\dim\, [e,\g(\tau,0)]\equiv \dim \g(\tau,2)\mod 2$. In view of our earlier remarks this yields $\g(\tau,2)=[e,\g(\tau,0)]$.

Let $L(\tau)=Z_{G}(\tau)$, a Levi subgroup of $P(\tau)$ with Lie algebra $\g(\tau,0)$. Since $[\g(\tau, 0),e]$ is contained in the tangent space $T_e\big((\Ad L(\tau))\,e\big)$, the $(\Ad L(\tau))$-orbit of $e$ is Zariski open in
$\g(\tau,2)$ and hence intersects with $\V(\tau,2)_{ss}\ne \varnothing$. This yields $e\in \H(\Delta)$.	
\end{proof}
\subsection{}\label{4.2} 
Given a connected reductive $\k$-group $H$ and a nilpotent element $e\in \Lie(H)$ (i.e. an unstable vector of the $(\Ad H)$-module $\Lie(H)$), we write $\hat{\Lambda}_H(e)$ for the set of all
cocharacters in $X_*(H)$ optimal for $e$ in the sense of the  Kempf--Rousseau theory. By [\cite{Hes78}, Theorem~7.2], this set of cocharacters does not depend on the choice of an $H$-invariant 
$\mathbb{R}_{\ge 0}$-valued norm mapping on $X_*(H)$. 

Let $\tau=\tau_\Delta$, where $\Delta\in\mathfrak{D}_{G}$, and let $e\in \g(\tau,2)$ be such that $G_{e}\subset P(\tau)$. As $e$ is a $G$-unstable vector of $\g$ it affords an optimal cocharacter $\tau'\in X_*(G)$
with the property that $e\in\g(\tau',\ge 2)$; see [\cite{CP}] for detail. Since $\Ad \tau(\k^\times)$ preserves the line $\k e$ and the cocharacter $\tau'$ is optimal for any nonzero scalar multiple of $e$, it follows from the main results of the Kempf-Rousseau theory that $\Ad \tau(\k)$ normalises the optimal parabolic subgroup $P(\tau')$ of $e$; see [\cite{P03}, Theorem~2.1(iv)], for example. 
Since
the latter is self-normalising and contains $G_{e}$ we have the inclusion $\tau(\k^\times)G_{e}\subset P(\tau')$. 
Hence $N_e:=N_{G}(\k e)=\tau(\k^\times)G_{e}$ is a subgroup of $P(\tau')$.
It follows from [\cite{Bor}, 11.14(2)] (applied to tori) that for any maximal torus $D$ of $G_{e}$ there is a maximal torus $\tilde{D}$ of $N_e$ such that $D\subseteq \tilde{D}\cap G_{e}$. Since 
$\tau(\k^\times)$ is contained in a maximal torus of $N_e$ and  
all maximal tori of $N_e$ are conjugate, this shows that $G_{e}$ contains a maximal torus which commutes with $\tau(\k^\times)$; we call it $T_0$. 

 As $N_e\subset P(\tau')$, there is a maximal torus $T'$ of $P(\tau')$ which contains the maximal torus $\tau(\k^\times)T_0$ of $N_e$. Since $L:=Z_{G}(T_0)$ is a Levi subgroup of $G$, there exists $g\in G$  such that
 $gLg^{-1}$ is a standard Levi subgroup $L$ of $G$.
 Replacing $e$ and $\tau$ by $(\Ad g)\,e$ and $g\tau g^{-1}$  we may assume without loss that $L$ is a standard Levi subgroup of $G$. 
 By the main results of the Kempf--Rousseau theory (and the description of Hesselink strata in [\cite{CP}]), there exists a unique cocharacter $\tau''\in \hat{\Lambda}_{G}(e)\cap X_*(T')$ conjugate to $\tau'$ under $P(\tau')$ and
 such that $e=\sum_{i\ge 2}\,e_i$ where $e_i\in  \g(\tau'', i)$ and $e_2\in\V(\tau,2)_{ss}$; see [\cite{P03}, Theorem~2.1(iii)].  

 By construction, $T'$ is a maximal torus of $L$ containing $\tau(\k^\times)$ and $\tau''(\k^\times)$. Furthermore, $e\in\g^{\Ad T_0}=\,\Lie(L)$. Let $L'=\mathcal{D}L$, the derived subgroup of $L$, and $\mathfrak{l}'=\Lie(L')$. It is immediate from Jacobson's formula for $[p]$-th powers in the restricted Lie algebra $(\mathfrak{l}, [p])$ that $e\in\mathfrak{l}'$. 
\subsection{}\label{4.3} Let $d=\dim T_0$ and $r=\ell - d$ where $\ell={\rm rk}\, G$. In this subsection we begin to investigate the case where $d>0$, i.e. the case where  $e$ is not distinguished in $\g$.
The group $L'=\mathcal{D}L$  is semisimple and $T'\cap L'$ contains a maximal torus of $L'$; we call it $T_1$. The subtorus $T_0\cdot T_1$ of $T'$ being self-centralising, it must be that $T'=T_0\cdot T_1$.  As the central subgroup $T_0\cap L'$ of $L'$ is finite, we have  a direct sum decomposition 
\begin{equation}\label{E2}
X(T')_\mathbb{Q}\,=\,X(T_0)_\mathbb{Q}\oplus X(T_1)_\mathbb{Q}
\end{equation}
of the $\mathbb{Q}$-spans of $X(T_0)$ and $X(T_1)$ in $X(T')_\mathbb{Q}=X(T')\otimes_{\Z}\mathbb{Q}$. This shows that $${\rm rk}\,L'=\dim T_1=
\ell-\dim_{\mathbb Q} X(T_0)_{\mathbb Q}=\ell-d=r.$$
Since we identify the dual spaces 
$X(T')_\mathbb{Q}$ and $X_*(T')_\mathbb{Q}=X_*(T')\otimes_{\Z}\mathbb{Q}$ by means of a $W$-invariant inner product $(\,\cdot\,|\,\cdot\,)$ we have that $\eta(\mu(t))=t^{(\mu\,|\,\eta)}$ for all  $\eta\in X(T)$, $\mu\in X_*(T')$, and $t\in \k^\times$.
As the subgroup $N_{L}(T')/T'$ of $W$ acts trivially on $X(T_0)$ and has no nonzero fixed points on $X(T_1)$, the $\mathbb{Q}$-spans $X_*(T_0)_\mathbb{Q}$ and $X_*(T_1)_\mathbb{Q}$ are orthogonal to each other with respect to $(\,\cdot\,|\,\cdot\,)$.
\begin{lemma}\label{L2}
Suppose $\nu\in X_*(T')$ is such that $e\in \g(\nu,2)$ and $(G_{e})^\circ\subset P(\nu)$. If the group $(G_e)^\circ/R_u(G_{e})$ is semisimple then $\nu=\tau$.	
\end{lemma} 
\begin{proof}
As $T_0\subseteq T'\cap G_e$ is a maximal torus of $G_e$ it must be that
$T_0=(T'\cap G_{e})^\circ$. As $\nu^{-1}(t)\cdot\tau(t)\in T'\cap G_{e}$ for all $t\in\k^\times$ 
we have that $\nu-\tau\in X_*(T_0)$.  As $R_u(G_{e})$ is a unipotent group, the torus $T_0\subset G_{e}$ maps isomorphically onto  a maximal torus of the semisimple group  $S_{e}:=(G_{e})^\circ/R_u(G_{e})$. It follows that ${\rm rk}\,S_{0}=\dim T_0=d.$
We identify $T_0$ with its image in $S_{e}$. As
$\nu(\k^\times)\subset N_e$ normalises $R_u(G_{e})$, it acts on $S_{e}$ by rational automorphisms. 

It is straightforward to see that the connected algebraic group
$\tilde{S}_{e}:=\nu(\k^\times)S_{e}$ is reductive and $\nu(\k^\times)T_0$ is a maximal torus of $\tilde{S}_{e}$. So, if $\bar{\gamma} \in X\big(\nu(\k^\times)T_0\big)$ is a root of $\tilde{S}_{e}$ then so is $-\bar{\gamma}$. On the other hand, our assumption on $\nu$ implies that $\Ad \nu$ has nonnegative weights on $\Lie(S_{e})$. This entails that $\nu(\k^\times)$ is a central torus of $\tilde{S}_{e}$. Repeating this argument with $\tau$ in place of $\nu$ we deduce that the torus $\nu(\k^\times)$ is central in $\tilde{S}_{e}$ as well.    

 Since ${\rm rk}\,S_{e}=d$ there are $\mathbb{Q}$-independent
weights $\gamma_1,\ldots, \gamma_d\in X(T_0)$ which serve as a basis of simple roots for the root system of $\tilde{S}_{e}$ with respect to $T_0$. 
As $\dim X(T_0)_\mathbb{Q}=d$, they must form a basis of the vector space $X(T_0)_\mathbb{Q}$.
On the other hand, the preceding discussion shows that
$\gamma_i(\nu(t))=
\gamma_i(\tau(t))=1$ for all $t\in\k^\times$ and $i\le d$. Our identification of $X(T')_\mathbb{Q}$ and $X_*(T')_\mathbb{Q}$ now yields that $\nu-\tau$ is orthogonal to 
$X(T_0)_\mathbb{Q}$ with respect to $(\,\cdot\,|\,\cdot\,)$. Since $\nu-\tau\in X_*(T_0)\subseteq X(T_0)_\mathbb{Q}$ this forces
$\nu-\tau=0$, completing the proof.
\end{proof}
\subsection{}\label{4.4} Recall from \ref{4.2} that our nilpotent element $e\in\mathfrak{l}'$ affords an optimal cocharacter $\tau''\in \hat{\Lambda}_{G}(e)\cap X_*(T')$ such that $e=\sum_{i\ge 2}\,e_i$ with $e_i\in \fl'(\tau'', i)$ and $e_2\in\V(\tau'',2)_{ss}$. 
Identifying $X_*(T')_\mathbb{Q}$ and $X(T')_\mathbb{Q}$ as in \ref{4.3} and using (\ref{E2}) we get $\tau''=\tau''_0+\tau''_1$ where $\tau''_0\in X(T_0)_\mathbb{Q}$ and $\tau''_1\in X(T_1)_\mathbb{Q}$. Since 
$(\tau''\,|\,\tau'')=(\tau''_0\,|\,\tau''_0)+(\tau''_1,\tau''_1)\ge (\tau''_0\,|\,\tau''_0)$
and $(\tau''_0\,|\,\gamma)=0$ for all $\gamma\in X(T_1)$, 
it follows from the optimality of $\tau''$ that $\tau''_0=0$; see [\cite{P03}. p.~348] for a similar (characteristic--free) argument. 

As $\tau''\in X_*(T')$
we thus obtain that $\tau''$ is an optimal cocharacter for $e$ contained in  $X_*(T_1)$.
Therefore, the orbit $\OO_L(e):=(\Ad L)\,e$ is contained in the Hesselink stratum $\H_L(\tau'')$
of the nilpotent cone $\NN(\fl')$, the variety  of all $(\Ad L)$-unstable vectors of $\fl'$.
Since $T_0$ is a maximal torus of $G_{e}$ the group $(L_e)^\circ$ is unipotent. In other words, $e$ is a distinguished nilpotent element of $\fl'$. 
\begin{lemma}\label{L3}
If the $(\Ad L)$-orbit $\OO_L(e)$ coincides with its stratum  $\H_L(\tau'')$ and the group
$(G_e)^\circ/R_u(G_{e})$ is semisimple, then $\tau$ is $L$-conjugate to $\tau''$ 
and  $e\in \H(\Delta)$.
 \end{lemma}
 \begin{proof}
 	Recall that $e=\textstyle{\sum}_{i\ge 2}\,e_i$ where $e_i\in\fl'(\tau'',i)$ and $e_2\in \V(\tau'',2)_{ss}$.
 As one of the $(\Ad L)$-orbits of $\H_L(\tau'')$ intersects with $\fl'(\tau'',2)$, our assumption on $\H_L(\tau'')$ implies that $e$ is $(\Ad L)$-conjugate to $e_2$. 
 Since $\tau''\in\hat{\Lambda}_{G}(e_2)$ we have $G_{e_2}\subset P(\tau'')$.
 As the group $(G_{e_2})^\circ/R_u(G_{e_2})\,\cong\, (G_{e})^\circ/R_u(G_{e})$ is semisimple, it follows from Lemma~\ref{L2} that $\tau$ is $L$-conjugate to $\tau''$. As $\tau''$ is 
 $G$-conjugate to $\tau_\Delta$ by our discussion in \ref{4.2}, we deduce that $e_2\in \H(\Delta)$. But then $e\in \H(\Delta)$ as wanted.
 	\end{proof}
 Let $e\in\NN(\g)$ and write $S_{e}$ for the factor-group $(G_{e})^\circ/R_u(G_{e})$. When ${\rm char}(\k)=0$, one can use 
 the tables in [\cite{Ca}, pp.~401--407] to quickly compile a full list of the nilpotent orbits $\OO(e)=(\Ad G)\,e$ 
 for which $S_{e}$ is a semisimple group. Then one can use [\cite{LT}] to find out that the same list is still valid in good characteristic. Since we are mainly concerned with
 the case where $p={\rm char}(\k)$ is very bad for $G$, we must rely instead on the following important classification result obtained in [\cite{LS}].	
 	\begin{prop}{\rm ([\cite{LS}])}\label{P2}
 	Suppose $G$ is exceptional and $e$ is not distinguished in $\g$. Then either $S_{e}$ is a semisimple group or $\OO(e)$  has one of the following labels:
 	\begin{enumerate}
 		
 		\smallskip
 		
 		\item[] ${\rm Type\, E_6}:$ $\ {\rm A_1^2,\, A_2A_1,\, A_2A_1^2,\,A_3,\,A_3A_1,\,D_4(a_1),\,A_4,\, A_4A_1,\,D_5(a_1),\,D_5}$.
 		
 		\medskip
 		
 		\item[] ${\rm Type\, E_7}:$ $\ {\rm A_2A_1,\,A_3A_2\,\,(\mbox{$p\ne 2$}),\,A_4,\,A_4A_1,\, D_5(a_1),\,E_6(a_1)}$.
 		
 		\medskip
 			
 		\item[] ${\rm Type\, E_8}:$ $\ {\rm A_3A_2\,\, (\mbox{$p\ne 2$}), \,\,A_4A_1,\,A_4A_1^2,\,D_5A_2\,\,(\mbox{$p\ne 2$}),\,D_7(a_2)}$,
 		
 		\smallskip
 		
 		\item[]\qquad\qquad \,\,   ${\rm E_6(a_1)A_1,\,D_7(a_1)\,\, (\mbox{$p\ne 2$}).}$
 	\end{enumerate}
 	If $G$ is of type ${\rm G}_2$ or ${\rm F}_4$ then $S_{e}$ is a semisimple 
 	group for any $e\in\NN(\g)$.
 	\end{prop}
\begin{proof}
	The statement
	is obtained by examining Tables~22.1.1\,--\,22.1.5 in [\cite{LS}].
	 One also observes in the process that if a nilpotent orbit $\OO(e)$ coincides with its Hesselink stratum then the type of $S_{e}$ is independent of the characteristic of $\k$. This is curious, but will not be required in what follows. 
\end{proof}
\subsection{}\label{4.5} 
It is well known that  $\dim\g_{e}\ge \dim G_{e}$ for any $e\in \g$, and if the equality
$\dim\g_{e}=\dim G_{e}$ holds, the orbit $\OO(e)$ is called {\it smooth}.
For all exceptional types, the smooth nilpotent orbits of $\g$ can be determined from Stewart's tables [\cite{St}] which record the Jordan blocks of all $\ad e$ 
with $e\in\NN(\g)$. We note that the representatives of nilpotent orbits used in Stewart's tables are compatible with those in [\cite{LS}, Tables~12.1, 13.3 and 14.1].
\begin{rem}\label{R3} Suppose $G$ is exceptional.
Although Liebeck and Seitz do not discuss Hesselink strata in [\cite{LS}], they can be spotted as follows. For each representative $e\in\NN(\g)$ listed in the tables of [\cite{LS}] there exists a cocharacter
$\mu\in X_*(G)$ conjugate to $\tau_\Delta$ with $\Delta\in\mathfrak{D}_{G}$ and such that $e\in \g(\mu,\ge 2)$ and $G_{e}\subset P(\mu)$. Of course, in characteristic $2$ and $3$ there are a few special cases where two representatives of different orbits, say $e$ and $\tilde{e}$, are attached to the same $\mu$ (and the same $\Delta$). 
If there is no $\tilde{e}$, then $e$ is homogeneous (that is, lies in $\g(\mu,2)$) and
$\overline{(\Ad P(\mu))\,e}=\g(\mu,\ge 2)$. Then the orbit $(\Ad Z_{G}(\mu))\,e$ is dense in $\g(\mu,2)$ and the description of Hesselink strata in [\cite{CP}] implies that
$e$ is $Z_{G}(\mu)$-semistable. This is an excellent case since $\H(\Delta)=\OO(e)$ is a single orbit and $\dim G_{e}=
\dim \g(\mu,0)+\dim\g(\mu,1)$ is independent of $p$.

If $\tilde{e}$ does exist then one may assume without loss that $e\in \g(\mu, 2)$ and $\tilde{e}=e+e_\beta$ for some root vector $e_\beta\in\g(\mu,d)$ with $d\ge 2$. Moreover, $\tilde{e}$ always lies in the new orbit with label $(\Delta)_p$ and
$(\Ad P(\mu))\,\tilde{e}$ is Zariski dense in $\g(\mu,\ge 2)$. Then $\dim G_{\tilde{e}}=
\dim \g(\mu,0)+\dim\g(\mu,1)$, but $\dim G_{e}> \dim G_{\tilde{e}}$.

If $d>2$ then it is still true that the orbit $(\Ad Z_{G}(\mu))\,e$ is dense in $\g(\mu,2)$. Therefore, both $e$ and $\tilde{e}$ lie in the stratum $\H(\Delta)$.
In fact, it follows from the main results of [\cite{LS}] that $\H(\Delta)=\OO(e)\cup \OO(\tilde{e})$. So this case is not too bad either.

In order to describe the strata of $\H(\Delta)$ explicitly, one has to clarify the remaining (problematic) case where $\tilde{e}=e+e_\beta$ and $d=2$. Here $\g(\mu,2)$ contains both $e$ and $\tilde{e}$, but the orbit $(\Ad Z_{G}(\mu))\,e$ is no longer dense in $\g(\mu,2)$. So we cannot conclude at this point that $e$ is $Z_{G}(\mu)$-semistable.  
However, since $\mu$ is $G$-conjugate to $\tau_\Delta$ and it is shown in [\cite{LS}] that $G_{e}\subset P(\mu)$,
we see that $e\in {\rm LX}(\Delta)$. Therefore, the main results of [\cite{LS}] together with Theorem~\ref{main} 
provide a very satisfactory description of the Hesselink strata of $\NN(\g)$ for $G$ exceptional. 
Namely, if there is no $\tilde{e}$ then $\H(\Delta)=\OO(e)$ and if $\tilde{e}$ does exist then
$\H(\Delta)=\OO(e)\cup \OO(\tilde{e})$.
\end{rem} 
\begin{rem}\label{R4} Suppose $e$ is as in Remark~\ref{R3} and $\OO(e)=\H(\Delta)$.  Then $\dim G_{e}=\dim \g(\mu,0)+\dim\g(\mu,1)$ is independent of the characteristic of $\k$. Since the representatives $e$ used in [\cite{St}] agree with those of [\cite{LS}] one can compute $\dim G_{e}$ by counting the number of Jordan blocks of $\ad e$ under the assumption that $p\gg 0$. If the number obtained coincides with the actual number of Jordan blocks of $\ad e\in\End \g$ then the orbit $(\Ad G)\,e$ is smooth. 
	\end{rem}
Thanks to Proposition~\ref{P2} and Lemma~\ref{L1} we can reduce proving statement (\ref{S1}) to a much smaller number cases.
\begin{lemma}\label{L4}
	Suppose $\tau\in X_*(G)$ is conjugate to $\tau_\Delta$ with $\Delta\in\mathfrak{D}_{G}$ and $e\in\g(\tau,2)$  is such that $G_{e}\subset P(\tau)$. It the group  $S_{e}$ is not semisimple, then $e\in\H(\Delta)$.
\end{lemma}
	\begin{proof} Since $S_{e}$ is not semisimple, $e$ lies in one of the orbits 
listed in Proposition~\ref{P2}. Our discussion in Remark~\ref{R3} (based on results of [\cite{LS}]) shows that any such orbit $\OO(\Delta)$ coincides with the corresponding Hesselink stratum $\H(\Delta)$. 
Using [\cite{St}, Tables~10, 11, 12] and the method described in Remark~\ref{R4} one now checks directly that
$\dim \g_{e}=\dim G_{e}$ unless $G$ is of type ${\rm E_6}$, $p=2$, and $\OO(e)$ is one of 
$\OO({\rm A_3A_1})$ or $\OO({\rm D_5})$. If $\dim \g_{e}=\dim G_{e}$ then applying Lemma~\ref{L1} gives $e\in\H(\Delta)$.

Now suppose $G$ is of type ${\rm E_6}$ and $p=2$. In the remaining two cases $e$ is a regular nilpotent element of a Levi subalgebra of type ${\rm A_3A_1}$ or ${\rm D_5}$, hence no generality will be lost by assuming that $e$ is  as in [\cite{LT}, pp.~85, 91]. 

We first suppose that $e\in\OO({\rm A_3A_1})$. Then $e=e_1+e_3+e_4+e_6$.
Since $e\in\g(\tau, 2)$ and all maximal tori of $N_e$ are conjugate we may also assume 
that $\tau=(2,r,2,2,s,2)$ for some $r,s\in\Z$. If $\tilde{\alpha}=122321$, the highest root of $\Sigma$ with respect to $\Pi$, then $\{x_{\pm \tilde{\alpha}}(t)\,|\,\,t\in\k\}$ generate a subgroup of type ${\rm A_1}$
in $G_{e}$.  Since $G_{e}\subset P(\tau)$ it must be that $(\tau\,|\,\tilde{\alpha})=0$ forcing $14+2(r+s)=0$. Therefore, $\tau=(2,r,2,2,-7-r,2)$. Since
$x_{-\alpha_5}(t) \in G_{e}$ for all $t\in\k$ we also have $7+r\ge 0$. Let $\beta={\alpha_3+\alpha_4+\alpha_5+\alpha_6}$ and $\beta'={\alpha_1+\alpha_3+\alpha_4+\alpha_5}$. Then $\beta\pm\beta'\not\in\Sigma$ and $x_\beta(t)x_{\beta'}(t)\in G_{e}$ for all $t\in\k$ yielding $e_\beta+e_{\beta'}\in\Lie(G_{e})$. Hence $6-(7+r)\ge 0$, i.e. $r\le -1$. As a result, $r\in\{-7,-6,-5,-4,-3,-2,-1\}$.
For $r$ in this set we denote by $\tau_r$ the $W$-conjugate of $\tau=(2,r,2,2,-7-r,2)$ contained in $X_*^{+}(T)$. 

It is straightforward to check that $\tau_{-7}=(0,1,1,0,1,2)$, $\tau_{-6}=(1,1,0,0,1,2)$, $\tau_{-5}=(1,0,0,1,0,2)$,
$\tau_{-4}=(1,1,0,0,1,1)$, $\tau_{-3}=(0,1,1,0,1,0)$, $\tau_{-2}=(1,1,1,0,0,1)$, and $\tau_{-1}=(2,0,0,1,0,1)$. Using [\cite{Ca}, p.~402] we now observe that only $\tau_{-3}$
has form $\tau_{\Delta}$ with $\Delta\in\mathfrak{D}_G$. Furthermore,
$\Delta={\rm A_3A_1}$. Since $\OO({\rm A_3A_1})=\H({\rm A_3A_1})$ by Remark~\ref{R3}, we get $e\in\H(\Delta)$.

Finally, suppose $e\in\OO({\rm D_5})$. Then $e=e_1+e_2+e_3+e_4+e_5$. As $e\in\g(\tau,2)$ and all maximal tori of $N_e$ are conjugate we may assume that
$\tau=(2,2,2,2,2, r)$ for some $r\in\Z$. 
Let $\gamma=\alpha_3+\alpha_4+\alpha_5+\alpha_6$ and $\gamma'=\alpha_2+\alpha_4+\alpha_5+\alpha_6$. Then $\gamma\pm\gamma'\not\in\Sigma$ and $x_{-\gamma}(t)x_{-\gamma'}(t)\in G_{e}$ for all $t\in\k$. It follows that $e_{-\gamma}+e_{-\gamma'}\in\Lie(G_{e})$. Hence $r\le -6$. Now let $\delta=\alpha_1+\alpha_2+2\alpha_3+2\alpha_4+\alpha_5+\alpha_6$ and $\delta'=\alpha_1+\alpha_2+\alpha_3+2\alpha_4+2\alpha_5+\alpha_6$. Then again $\delta\pm\delta'\not\in\Sigma$ and $x_{\delta}(t)x_{\delta'}(t)\in G_{e}$ for all $t\in\k$, so that $e_\delta+e_{\delta'}\in\Lie(G_{e})$. It follows that $r\ge -14$.

As a result, $r\in \{-14,-13,-12,-11,-10,-9,-8,-7,-6\}$. Let $\tau_r$ be the $W$-conjugate of $\tau=(2,2,2,2,2,r)$ contained in $X_*^{+}(T)$. Direct computations show that $\tau_{-14}=(2,0,2,2,2,2)$, $\tau_{-13}=(2,1,2,1,1,1)$,
$\tau_{-12}=(2,2,2,0,2,0)$, $\tau_{-11}=(2,2,1,1,1,1)$, $\tau_{-10}=(2,2,0,2,0,2)$,
 $\tau_{-9}=(1,2,1,1,1,2)$, $\tau_{-8}=(0,2,2,0,2,2)$, $\tau_{-7}=(1,1,1,1,2,2)$,
$\tau_{-6}=(2,0,0,2,2,2)$. Using [\cite{Ca}, p.~402] we observe that only $\tau_{-10}$
has the form $\tau_{\Delta}$ with $\Delta\in\mathfrak{D}_G$. Moreover,  $\Delta={\rm A_5}$. Since $\OO({\rm D_5})=\H(\rm {D_5})$ by Remark~\ref{R3}, we get $e\in\H(\Delta)$, completing the proof.
\end{proof}
\subsection{}\label{4.6}
We are now in a position to reduce proving statement~(\ref{S1}) to the case where $e\in\g(\tau, 2)$ is distinguished in $\g$ and $\OO(e)$ is a proper subset of its Hesselink stratum.

Suppose $e\in \g(\tau,2)$ is not distinguished in $\g$. By \ref{4.2}, the group $G_{e}$ contains a maximal torus $T_0$ commuting with $\tau(\k^\times)$. We also know that the centraliser $L=Z_{G}(T_0)$ is a standard Levi subgroup of $G$ containing a maximal torus $T'$ such that $\tau(\k^\times)T_0\subseteq T'$. Let $L'=\mathcal{D}L$ and $\fl'=\Lie(L')$.
By our discussion in \ref{4.2}, there exists $\tau''\in\hat{\Lambda}_{G}(e)\cap X_*(T')$ such that $e=\sum_{i\ge 2} e_i$ where 
$e_i\in\fl'(\tau'',i)$ and $e_2\ne 0$. If $e$ is $(\Ad L)$-conjugate to $e_2$ then $e\in \H(\Delta)$ by Lemma~\ref{L3}. Therefore, we may assume that $e$ and $e_2$ lie in different orbits of $\H_L(\tau'')$. 
Thanks to Lemma~\ref{4.4} we may also assume that the group $S_{e}$ is semisimple.

If follows from [\cite{Hes78a}, Table~4] that if $H$ is a group of type ${\rm D}_r$ with $r\in\{4,5\}$ then every Hesselink stratum of $\Lie(H)$ is a single $(\Ad H)$-orbit.
Since the same holds for $H$ of type ${\rm E_6}$ and ${\rm A}_r$ with $r\ge 1$, the proper Levi subgroup $L$ of $G$ must have a component of type ${\rm D_6}$, ${\rm D_7}$, ${\rm E_7}$, ${\rm B_2}$, ${\rm B_3}$ or ${\rm C_3}$, where in the last three cases $G$ is a group of type ${\rm F_4}$. 
Thanks to [\cite{LS}, Tables~22.1.1\,--\,22.1.5] we now see that the case we consider can occur only when $p=2$ and $G$ is not of type ${\rm E_6}$ or ${\rm G_2}$.  
	
Suppose $p=2$ and $G$ is of type ${\rm E_7}$. As $\H_L(\tau'')$ contains more than one orbit, the group $L'$ must have type ${\rm D_6}$. By [\cite{Hes78a}, Table~4], the Lie algebra $\fl'$ has a unique new distinguished nilpotent orbit, and the above discussion shows that this new orbit, labelled ${(\rm A_3A_2)_2}$ in [\cite{LS}], must coincide with $\OO_L(e)$.  
By [\cite{LS}, Lemma~12.6], there exists $\mu\in X_*(L')$ such that $e\in\fl'(\mu,2)$ and the orbit $((\Ad P_L(\mu))\,e$ is dense in $\fl'(\mu, \ge 2)$. From this it is immediate that $e$ is a $Z_L(\mu)$-semistable vector of $\fl'(\mu,2)$, so that $\mu\in\hat{\Lambda}_L(e)$. Since
$\hat{\Lambda}_L(e)$ contains $\hat{\Lambda}_{G}(e)\cap X^*(T')$ and $e_2\ne 0$, the cocharacters $\mu$ and $\tau''$ are $L$-conjugate. But then $G_{e}\subset P(\nu)$ and applying Lemma~\ref{L2} with $\nu=\mu$ we get $\mu=\tau$. As a result, $e\in\H(\Delta)$.

Suppose $p=2$ and $G$ is of type ${\rm E_8}$. Then $L'$ has type ${\rm D_6}$, ${\rm D_7}$ of ${\rm E_7}$. If $L'$ has type ${\rm D_6}$ we can repeat verbatim the argument from the previous paragraph to conclude that $e\in \H(\Delta)$. Suppose $L'$ is of type ${\rm D_7}$. By [\cite{Hes78a}, Table~4], the Lie algebra $\fl'$ has two new orbits, denoted by ${(\rm A_3A_2)_2}$ and ${(\rm D_4A_2)_2}$ in [\cite{LS}, Table~15.3]. Since the orbit $\OO_{L'}({(\rm A_3A_2)_2})$ is not distinguished in $\fl'$ it must be that 
$e\in\OO_{L'}({(\rm D_4A_2)_2})$. As [\cite{LS}, Lemma~12.6] is also applicable  
for $e\in\OO_{L'}({(\rm D_4A_2)_2})$ , we can argue as in the second part of the previous paragraph to conclude that
$e\in \H(\Delta)$.
\subsection{}\label{4.7}  Retain the assumptions of \ref{4.6} and
suppose that $p=2$ and $L'$ is of type ${\rm E}_7$. Since $e$ is distinguished in $\fl'$ and $\OO_{L'}(e)\ne \OO_{L'}(e_2)$, it follows from [\cite{LS}, Tables~22.1.2] that $e\in\OO_{L'}(({\rm A_6})_2)$. 
Conjugating $e$ by a suitable element of  $L'$ we may assume that  
\begin{equation}\label{E3}
e=e_{56}+e_{67}+e_{134}+e_{234}+e_{345}+e_{245}+e_{123^24^25},
\end{equation}
 where all summands $e_\gamma$ involved in (\ref{E3}) are root vectors of $L'$ with respect to $T'$; see [\cite{LS}, Table~14.1].  Let $W'=N_{L'}(T')/T'$, the Weyl group of $L'$.
 It is easy to check that the cocharacter $\mu=(-2,-2,-2,6,-2,4,-2)$ has the property that $e\in\fl'(\mu,2)$. 
 Since $e$ is distinguished in $\fl'$, the group $L'_e$ is unipotent and $N_{L'}(\k e)=\mu(\k^\times)Z_{L'}(e)$, so that $\mu(\k^\times)$ is a maximal torus of $N_{L'}(\k e)$. Since $e\in \g(\tau,2)$ and $\tau\in X_*(T')$ this entails that $(\Ad \mu(t))\,x=(\Ad \tau(t))\,x$ for all $x\in\fl'$ and $t\in\k^\times$.
 
 A direct computation shows that $\mu$ is $W'$-conjugate to 
 $\tau_{\Delta'}=(2,0,0,2,0,0,2)$ which 
 corresponds to the distinguished $L'$-orbit with label ${\rm E_7(a_4)}$. The latter coincides with its Hesselink stratum $\H_{L'}(\Delta')$. From this it is immediate that $\fl'(\mu, 2)$ contains a Zariski open $Z_{L'}(\mu)$-orbit consisting of  $Z_{L'}(\mu)$-semistable vectors. Since $e\not\in \H_{L'}(\Delta')$ the orbit $(\Ad Z_{L'}(\mu)\,e$ is not dense in $\fl'(\mu,2)$. Let $U_{L'}(\mu):=R_u(P_{L'}(\mu))$. Then $P_{L'}(\mu)=Z_{L'}(\mu)\,U_{L'}(\mu)$. Besides,
 $\Lie(U_{L'}(\mu))=\fl'(\mu,\ge 2)$ and 
  $\Lie(Z_{L'}(\mu))=\fl'(\mu,0)$. 
  In type ${\rm E}_7$, the present case was investigated in [\cite{LS}, pp.~209].
In particular, it was shown there that $L'_e$ is a connected unipotent group of dimension
$19$ and $A:=L'_e\cap Z_{L'}(\mu)$ is a $1$-dimensional connected subgroup of $L'_e$ with the property that $\Lie(A)=\fl'_e(\mu,0)$. 

Straightforward computations show that  
$[e,\fl'(\mu, 4)]$ has codimension $1$ in $\fl'(\mu,6)$ and $[e,\fl'(\mu, 2i)]=\fl'(\mu,2i+2)$ for $i=1$ and all $i\ge 3$. 
Since the group $U_{L'}(\mu)$ is generated by the root elements $x_\alpha(t)\in L'$ with $(\mu|\alpha)\ge 2$ we have that
$(\Ad U_{L'}(\mu))\,e\subseteq 
e+\fl'(\mu,\ge 4).$
Since $[e,\fl'(\mu,2i)]\subset T_e(\Ad U_{L'}(\mu)\,e)$ for all $i\ge 1$, the preceding remark 
 yields that
$T_e(\Ad U_{L'}(\mu)\,e)$ has codimension $\le 1$ in $\fl'(\mu,\ge 4)$.
Therefore,
$\dim U_{L'}(\mu)_e=\dim \Lie(U_{L'}(\mu))-\dim\big(T_e(\Ad U_{L'}(\mu)\,e)\big)
	\le \dim \fl'(\mu,2)+1=18$
(one should keep in mind that $P_{L'}(\mu)$ is a distinguished parabolic subgroup of $L'$ and $\fl'(\mu,0)$ has dimension $17$). Hence $A\cdot U_{L'}(\mu)_e\subseteq  P_{L'}(\mu)_e$ has dimension $18$ or $19$. 

Since $L'_e$ is a connected group of dimension $19$ there are two possibilities one of which would be very bad for us: either $L'_e\subset P_{L'}(\mu)$  or
$L'_e\not\subset P_{L'}(\mu)$ and $T_e(\Ad U_{L'}(\mu)\,e)\,=\,\fl'(\mu,\ge 4)$.
In the second case we would have
$(\Ad U_{L'}(\mu))\,e=e+\fl'(\mu,\ge 4)$ by Rosenlicht's theorem [\cite{Ros}, Theorem~2]. 

Once again our main source of reference comes to the rescue: it is proved in [\cite{LS}, p.~208] that 
$L'_e$ contains the $1$-parameter unipotent subgroup 
$$U=\{x_{\alpha_1}(c)x_{\alpha_1+\alpha_3}(c^2)x_{\alpha_2}(c)x_{\alpha_5}(c)x_{\alpha_7}(c)\,|\,\,c\in\k\}.$$
The Lie algebra of $U$ is spanned by $v=e_1+e_2+e_5+e_7\in \fl'(\mu,-2)$ (and one can check directly that $[e,v]=0$). Therefore, $L'_e\not\subset P_{L'}(\mu)$. As a byproduct we obtain that the orbit $(\Ad P_{L'}(\mu))\,e$ has codimension $1$ on $\fl'(\mu,\ge 2)$.

In type ${\rm E_8}$ the present case was investigated in [\cite{LS}, pp.~247, 248] where the element in (\ref{E3}) was replaced by its $(\Ad G)$-conjugate
$$e'=e_1+e_3+e_4+e_5+e_6+e_7+e_{1234^25^267}.$$ The cocharacter
$\mu$ used above was replaced by $\mu'=(2,-14,2,2,2,2,2,-3)$. (In [\cite{LS}, p.~248], the torus $\mu'(\k^\times)$ is denoted by $\widetilde{T}$.) Direct computations show  that
there is $w\in W({\rm E_8})$ such that $w(\mu')=(0,0,0,1,0,1,0,2)=\tau_{\Delta'}$ where $\Delta'={\rm E_7(a_4)}$. Let $\tilde{\alpha}=23465432$, the highest root of $\Sigma$ with respect to $\Pi$, and denote by $\tilde{\alpha}^\vee$ the corresponding coroot in $X_*(T)$, so that $(\Ad\,\tilde{\alpha}^\vee(t))(e_\gamma)=t^{\la\gamma\,,\,\tilde{\alpha}^\vee\ra} e_\gamma$ for all $t\in\k^\times$ and $\gamma\in\Sigma$. The adjoint action of $\tilde{\alpha}^\vee(\k^\times)$ endows $\g$ with a short $\Z$-grading $\g= \g_{-2}\oplus\g_{-1}\oplus\g_0\oplus\g_1\oplus\g_2$ such that $\dim\g_{\pm 2}=\k e_{\pm\tilde{\alpha}}$ and $\g_0=\Lie(\tilde{L})$ where $\tilde{L}=Z_G(\tilde{\alpha}^\vee)$.

The derived subgroup $\tilde{L}'$ of $\tilde{L}$ has type ${\rm E_7}$ and $\tilde{L}=
T_0\,\tilde{L}'$ where $T_0=\tilde{\alpha}^\vee(\k^\times)$ is a $1$-dimensional central torus of $\tilde{L}$. 
The type ${\rm A_1}$ subgroup $S$ generated by $x_{\pm\tilde{\alpha}}(\k)$ and $T_0$ commutes with 
$\tilde{L}'$. We pick $\sigma\in N_S(\tilde{\alpha}^\vee(\k^\times))$
such that $\sigma(\tilde{\alpha}^\vee)=-\tilde{\alpha}^\vee$. Then $(\Ad \sigma)(\g_1)=\g_{-1}$ implying that $\g_1$ and $\g_{-1}$ are isomorphic as $(\Ad \tilde{L}')$-modules (both modules have dimension $56$ and are irreducible over $\tilde{L}'$). The parabolic subgroup $P(\tilde{\alpha}^\vee)=\tilde{L}\, Q$,
where $Q=R_u(P(\tilde{\alpha}^\vee))$, has the property that $\Lie(Q)=\g_1\oplus\g_2$ and  $\mathcal{D}Q=x_{\tilde{\alpha}}(\k)$. 

Computations in [\cite{LS}, p.~248] show that $\tilde{T}=\mu'(\k^\times)\subset 
T\cap \tilde{L}'$, and $Q\cap G_{e'}$ contains an $8$-dimensional abelian connected unipotent subgroup $V=\prod_{i=1}^8\,V_i$ such that each $V_i$ ia $1$-dimensional subgroup of $V$ normalised by the torus $\mu'(\k^\times)$. The $V_i$'s are described explicitly 
in {\it loc.\,cit.} and it is straightforward to check that $\mu'(\k^\times)$ acts on $V_1$, $V_2$, $V_3$, $V_4$, $V_5$, $V_6$, $V_7$ and $V_8$ with weights $11$, $5$, $3$, $9$, $9$, $1$, $3$ and $7$, respectively. As $\mu'(\k^\times)\subset \tilde{L}'$  commutes with $S$ and $e'\in \Lie(\tilde{L}')$, the group $(\Ad \sigma)(V)=
\prod_{i=1}^8\,(\Ad\sigma)(V_i)\subset G_{e'}$ has the same properties. As a consequence,
 both $V$ and $(\Ad \sigma)(V)$ are contained in $P(\mu')_{e'}$. 
 
 As $e'\in \Lie(\tilde{L}')$ we also have that $S\subset P(\mu')_{e'}$.
 Since $e$ and $e'$ are $(\Ad G)$-conjugate, our discussion at the beginning of this subsection shows that $P_{\tilde{L}'}(\mu')_{e'}$ has codimension $1$ in $\tilde{L}'_{e'}$.
 In particular, $G_{e'}\not\subset  P(\mu')$.
 
 On the other hand, we have that 
 $$\Lie(V)\oplus (\Ad \sigma)(\Lie(V))\oplus \Lie(S) \oplus\Lie(P_{\tilde{L}'}(\mu')_{e'})\subseteq \Lie(P(\mu')_{e'}),$$  and the subspace on the left has dimension $8+8+3+18=37$. As $G_{e'}$ is a group of dimension $38$ by [\cite{LS}, Table~22.1.1], we have that $P(\mu')_e$ has codimension $1$ in $G_{e'}$
  This, in turn, implies that 
 the orbit $(\Ad P(\mu')\,e'$ has codimension $1$ in $\g(\mu',\ge 2)$ and
 $(\Ad R_u(P(\mu'))\,e'\,=\,e'+\g(\mu',\ge 4)$ (since this will not be required in what follows we omit the details).

 We mention for completeness that if $\nu\in X_*(T)$ is such that $e'\in\g(\nu,2)$ and $G_{e'}\subset P(\nu)$ then the root vectors
 $e_{\pm\tilde{\alpha}}\in\Lie(G_{e'})$ must have weight $0$ with respect to $\nu$. Since 
 $G_{e'}\not\subset P(\mu')$, we have thus excluded the only available option for $\nu$, namely, $\nu=\mu'$. Therefore, $\OO({\rm (A_6)_2)}\cap(\cup_{\Delta\in\mathfrak{D}_G}\,\g_2^{\Delta,!})=\varnothing$, that is, the present case cannot occur in types ${\rm E_7}$ and ${\rm E_8}$.
 \subsection{}\label{4.8} Retain the assumptions of \ref{4.6} and suppose that $p=2$ and 
$L'$ is of type ${\rm F_4}$. Since this case has already been treated in [\cite{Vog}], our goal here is to offer a different proof. In view of our remarks in \ref{4.6} we may assume that the Hesselink stratum of $\NN(\fl')$ has more than one orbit and $L'$ is of type
${\rm B}_2$, ${\rm B_3}$ or ${\rm C_3}$. The tables in [\cite{Hes78a}] show that in these cases $\fl'$ has
a unique new distinguished nilpotent orbit. 

The corresponding nilpotent orbits of $\g$ are denoted in [\cite{LS}, Table~22.1.4] by $({\rm \tilde{A}_1})_2$, $({\rm B_2)_2}$ and $({\rm \tilde{A}_2})_2$, respectively.
Parts (i) and (iii) of the proof Lemma~16.9 in [\cite{LS}] show that in the first two cases there exists $\mu\in X_*(L')$ such that $e\in\fl'(\mu,2)$ and the orbit $((\Ad P_L(\mu))\,e$ is dense in $\fl'(\mu, \ge 2)$. As before, this enables us to deduce  that $e$ is $Z_L(\mu)$-semistable in $\fl'(\mu,2)$, so that $\mu\in\hat{\Lambda}_L(e)$. As
$\hat{\Lambda}_L(e)$ contains $\hat{\Lambda}_{G}(e)\cap X^*(T')$ and $e\in\g(\tau,2)$,
applying Lemma~\ref{L2} 
gives  $\mu=\tau$. Hence  $e\in\H(\Delta)$.

Suppose $e\in\OO(({\rm \tilde{A}_2})_2)$. In this case, we may assume that
$e=e_{0121}+e_{1111}+e_{2342}$ and $T'$ is the torus used in the proof of part (ii) 
of  [\cite{LS}, Lemma~16.9]. Let $\tau=(a_1,a_2,a_3,a_4)$ where $a_i\in \Z$, As 
$e\in\g(\mu,2)$ we have $a_2+2a_3+a_4=2$, $a_1+a_2+a_3+a_4=2$ and $2a_1+3a_2+4a_3+2a_4=2$. Solving this system of linear equations gives 
$\tau=(r,-2-r, r,4)$ where $r\in \Z$. By [\cite{LS}, p.~274], the group $G_{e}$ contains
$x_{\pm \alpha_2}(t)$ and $x_{\alpha_1}(t)x_{\alpha_3}(t)$ for all $t\in \k$. Therefore,
$\Lie(G_{e})$ contains $e_{\pm \alpha_2}$ and $e_{\alpha_1}+e_{\alpha_3}$. Since $G_{e}\subset P(\tau)$,
we must have $-2-2r=0$ and $r\ge 0$. This shows that $\tau$ does not exist, that is, the present case cannot occur.
 \subsection{}\label{4.9} From now on we may assume that $e\in\g(\tau_\Delta,2)$ is a distinguished nilpotent element of $\g$. 
 If $\OO(e)=\H(\Delta')$ then there is $\mu\in\hat{\Lambda}_G(e)$ such that
 $e\in\g(\mu,2)$.  As  Lemma~\ref{L2}
is still applicable in the present case we get $\mu=\tau_\Delta$ forcing $e\in\H(\Delta)$.  
In other words, statement~(\ref{S1}) holds for $e$.
 
This shows that we may assume further that $\OO(e)\subsetneq \H(\Delta)$. 
Thanks to the classification results of [\cite{LS}] it remains to consider the case where $e\in\NN(\g)$ is
distinguished and non-standard  (see Remark~\ref{R3} for more detail). No such orbits exist when $G$ has type ${\rm E_6}$ and when $G$ is of type ${\rm G_2}$ and $p=2$. 

If $G$ is of type ${\rm E}_7$ then $p=2$ and $\OO(e)=\OO(({\rm A_6)_2})$. In \ref{4.7}, we have shown that this case cannot occur in our situation.

\begin{lemma}\label{L5}
Suppose $e\in \g(\tau_\Delta,2)$ is distinguished in $\NN(\g)$ and $G_e\subset P(\tau_\Delta)$.
If $e'\in \OO(e)\cap \g(\mu,2)$, where $\mu\in X_*(G)$, then there is $g\in G$ such that
$e'=(\Ad g)(e)$ and $\mu=g\tau_\Delta g^{-1}$. Consequently, $G_{e'}\subset P(\mu)$.
\end{lemma}
\begin{proof}
	Let $g\in G$ be such that $(\Ad g)\,e=e'$. Then $e'\in \g(g\tau_\Delta g^{-1},2)$ and $G_{e'}\subset P(g\tau_\Delta g^{-1})$.
	Since the group $(G_{e'})^\circ$ is unipotent and all maximal tori of $N_{G}(\k e')^\circ=\tau_{\Delta'}(\k^\times)\, (G_{e'})^\circ$ are conjugate, there is $u\in (G_{e'})^\circ$ such that $ug\tau_\Delta g^{-1}u^{-1}=\mu$. 
	Hence $(\Ad ug)(e)=(\Ad u)(e')=e'$
	and
	$G_{e'}=uG_{e'}u^{-1}\subset P(ug\tau_\Delta g^{-1}u^{-1})=P(\mu)$.
\end{proof}
It remains to investigate the case where $(\Sigma,p)$ is one of  $({\rm E_8}, 3)$,
 $({\rm E_8}, 2)$,  $({\rm F_4}, 2)$ or $({\rm G_2},3)$.
 If $(\Sigma, p)=({\rm E_8},3)$ then $\OO(e)=\OO({\rm (A_7)_3})$. 
 Recall our standing assumption that $e\in\g(\tau,2)$, where $\tau=\tau_\Delta$, and $G_e\subset P(\tau)$.
  By [\cite{LS}, Table~4.1], we may assume that $e$ is $(\Ad G)$-conjugate to 
 $$e''= e_{567}+e_{1234}+e_{1345}+e_{3456}+e_{2456}+e_{234^25}+e_{678}+e_{45678}.$$
 The cocharacter $\mu=(0,2,2,-2,2,0,0,2)\in X_*(T)$ has the  property that $e\in\g(\mu,2)$ and is $W$-conjugate to $\tau_{\Delta'}=(0,0,0,2,0,0,0,2)$ with $\Delta'={\rm E_8(b_6})$; see [\cite{LS}, p.~210]. Replacing $e''$ by its $N_G(T)$-conjugate, $e'$ say, we may assume  that $e'\in \g(\tau_{\Delta'}, 2)$.
 By Lemma~\ref{L5}, there is $g\in G$ such that $e'=(\Ad g)(e)$ and $\tau_{\Delta'}=g\tau_\Delta g^{-1}$. Since both $\tau=\tau_\Delta$ and $\tau_{\Delta'}$ lie in $X_*^{+}(T)$  it must be that $\Delta=\Delta'$.
 
 We claim that contrary to our standing assumption on $\tau$ the group $G_{e}$ is not contained in $P(\tau)=P(\tau_{\Delta'})$. Indeed, it follows from [\cite{LS}, Theorem~3.2] that the nonempty open subset 
$\V(\tau,2)_{ss}$ of $\g(\tau,2)$ contains the open $Z_{G}(\tau)$-orbit of $\g(\tau,2)$, we call it $V$. It has the property that
$(\Ad P(\tau))\,v$ is dense in $\g(\tau,\ge 2)$ for every $v\in V$.
If $x\in V$ then $\dim G_{x}=28$, whilst $\dim G_{e'}=30$ by [\cite{LS}, Table~22.1.1].
Since $e'\in\OO(({\rm A_7)_3})$ it must be that $e\not\in \V(\tau,2)_{ss}$. But then the orbit $(\Ad Z_{G}(\tau))\,e$  is not dense in $\g(\tau,2)$ implying that $[e,\g(\tau,0)]$ is a proper subspace of $\g(\tau,2)$. 

In the present case, the normalised Killing form $\kappa$ is non-degenerate and induces a perfect pairing between $\g(\tau,2)$ and $\g(\tau,-2)$. This yields that $\g_e(\tau,-2)=[e,\g(\tau,0)]^\perp$ is nonzero (a different proof can be found in [\cite{LS}, p.~211]). 
On the other hand, $\dim \g_{e}=\dim G_{e}=30$ by [\cite{St}, Table~10]. If $G_{e}\subset P(\tau)$ then $\g_{e}=\Lie(G_{e})$ is contained in 
$\g(\tau,\ge 0)=\Lie(P(\tau))$. As $\g_{e}(-2)\ne 0$, we reach a contradiction. This shows that the present case cannot occur. In other words, $\OO({(\rm A_7)_3})$ has no elements
contained in the union of $\g_2^{\Delta,!}$ with $\Delta\in \mathfrak{D}_G$.
\subsection{}\label{4.10} Suppose $(\Sigma,p)=({\rm E_8},2)$. In this case we have to consider the new distinguished nilpotent orbits, namely, $\OO({(\rm D_5A_2)_2})$, $\OO(({\rm D_7(a_1)_2})$ and $\OO(({\rm D_7)_2})$. Thanks to [\cite{LS}, Table~14.1] we may choose the following representatives in the respective cases
\begin{eqnarray*}
	e'&=&e_{12345}+e_{234^25}+e_{13456}+e_{23456}+e_{34567}+e_{24567}+e_{78}+e_{678},\\
	e'&=&e_5+e_{45}+e_{234^2567}+e_{13}+e_{2456}+e_{3456}+e_{78}+e_8,\\
	e'&=&e_1+e_{234}+e_{345}+e_{245}+e_{456}+e_{567}+e_{678}+e_{12345678}.
\end{eqnarray*}
It is easy to check that in the first case $e'\in\g(\tau_{\Delta'},2)$ where $\tau_{\Delta'}=(0,0,0,0,2,0,0,2)$ is attached to the orbit $\OO({\rm D_5A_2})$.
By Lemma~\ref{L5}, 
there is $g\in G$ be such that $e'=(\Ad g)\,e$ and $\tau_{\Delta'}=g\tau g^{-1}$.
Since $\dim G_{e'}=34=\dim \g(\tau_{\Delta'}, 0)$ by [\cite{LS}, Table~22.1.1] and $G_{e'}\subset P(\tau_{\Delta'})$ by the Lemma~\ref{L5}, the orbit $(\Ad P(\tau_{\Delta'}))\,e$ must be open in $\g(\tau_{\Delta'},\ge 2)$.
But then $e$ must lie in the open  $(\Ad Z_{G}(\tau_{\Delta'}))$-orbit of $\g(\tau_{\Delta'},2)$. As a consequence, $e$ belongs to the nonempty open subset $\V(\tau_{\Delta'},2)_{ss}$ of $\g(\tau_{\Delta'}, 2)$. Since $\tau=\tau_\Delta$ and $\tau_{\Delta'}$ are $G$-conjugate and lie in $X_*^{+}(T)$ we conclude that $\Delta'=\Delta$. 
Hence $e\in\H(\Delta)$.

In the second case, one checks that $e'\in \g(\tau_{\Delta'},2)$ where $\tau_{\Delta'}=(2,0,0,0,2,0,0,2)$ is attached to the orbit $\OO({\rm D_7(a_1)})$. 
By [\cite{LS}, Table~22.1.1], we have $\dim G_{e}=26=\dim\g(\tau_{\Delta'},0)$.
Since $(G_{e})^\circ$ is unipotent, applying Lemma~\ref{L5} and arguing as in the previous case we deduce that $\tau=\tau_\Delta$ and $\tau_{\Delta'}$ are $G$-conjugate and $e\in\V(\tau_{\Delta'},2)_{ss}$. Therefore, $e\in \H(\Delta)$ as wanted.

The case where $e\in \OO(({\rm D_7)_2})$ is more complicated. First we note that
$e'\in\g(\mu,2)$ where
$\mu=(2,-4,-4,10,-4,-4, 10, -4)\in X_*(T)$. One checks directly that $\mu$ is $W$-conjugate to $(2,0,0,2,0,0,2,2)=\tau_{\Delta''}$, where $\Delta''={\rm E_8(b_4)}$. Since both $\tau=\tau_\Delta$ and $\tau_{\Delta''}$ lie in $X_*^{+}(T)$, it follows from Lemma~\ref{L5}
that $\tau=\tau_{\Delta''}$ and there is $v\in \OO(({\rm D_7})_2)\cap \g(\tau,2)$ 
such that $G_v\subset P(\tau)$.
Let $$v'=e_{13}+e_{234}+e_{345}+e_{245}+e_{567}+e_{456}+e_7+e_8,$$ an element of $\g(\tau,2)$. By [\cite{LS}, Tables~13.3 and 22.1.1], we have that $v'\in\OO({\rm E_8(b_4)})$ and $\dim G_{v'}=18=\g(\tau,0)$. 
Since $G_{v'}\subset P(\tau)$ by [\cite{LS}, Theorem~15.1(ii)] the orbit $(\Ad P(\tau)\,v'$ is open in $\g(\tau,\ge 2)$. It follows that the orbit
$V':=(\Ad Z_{G}(\tau))\,v'$ is open in $\g(\tau,2)$. As $v\not\in \OO(v')$ we have that $v\not\in V'$. Hence $\dim\,(\Ad Z_{G}(\tau))\,v<\dim\,\g(\tau,2)$ and, as a consequence,  $[v,\g(\tau,0)]$ is a proper subspace of $\g(\tau,2)$. 

By [\cite{LS}, Table~13.4], the maps $\ad v'\colon\,\g(\tau,4)\to \g(\tau,6)$ and  $\ad v'\colon\,\g(\tau,8)\to \g(\tau,10)$ are not surjective (this also follows from the fact that
$0\ne (v')^{[2]}\in \g_{v'}(\tau,4)$ and $0\ne (v')^{[4]}\in \g_{v'}(\tau,8)$ which is easy to see directly by applying Jacobson's formula for $[p]$-th powers with $p=2$). Using the perfect pairings between $\g(\tau,i)$ and $\g(\tau,-i)$ induced by the normalised Killing 
form $\kappa$ (which is non-degenerate in the present case) one observes that $\g_{v'}(\tau,r)\ne 0$ for $r\in \{-6,-10\}$.
As $v\in\g(\tau,2)$ lies in the Zariski closure of $(\Ad Z_{G}(\tau))\,v'$,
the semi-continuity of the nullity of a rectangular matrix yields that $\g_{v}(\tau, -6)\ne 0$ and  $\g_{v}(\tau, -10)\ne 0$, whilst our earlier remarks entail that $\g_{v}(\tau, -2)=[v,\g(\tau,0)]^\perp\ne 0$. Therefore, $\dim \g_{v}(\tau,<0)\ge 3$. 

As $v\in\OO(e)=\OO(({\rm D_7)_2})$ it follows from [\cite{St}, Table~10] that $\dim \g_{v}=24$. If $G_{v}\subset P(\tau)$ then [\cite{LS}, Table~22.1.1] shows that $\Lie(G_{v})$ is a Lie subalgebra of dimension $22$ in $\g(\tau,\ge 0)$.
But then
$$\dim \g_{v}=\dim\g_{v}(\tau,<0) + \dim \g_{v}(\tau,\ge 0)\ge 3+22=25.$$
This contradiction shows that the present case does not occur, that is, $e\in \OO({(\rm D_7)_2})$ cannot appear as an
element of $\g_2^{\Delta,!}$ with $\Delta\in \mathfrak{D}$.
We thus conclude that statement~(\ref{S1}) holds in type ${\rm E_8}$.
\subsection{}\label{4.11} Suppose $(\Sigma, p)$ is one of $({\rm F_4},2)$ or $({\rm G_2}, 3)$. These cases have been treated in [\cite{Vog}] by computational methods. The argument below will provide an alternative proof. In type ${\rm F_4}$, we only need to consider the non-standard distinguished orbits with labels ${\rm (\tilde{A}_2A_1)_2}$, ${\rm (C_3(a_1))_2}$ and ${\rm (C_3)_2}$.
Indeed,  [\cite{LS}, Theorem~16.1(ii)] implies that every standard distinguished orbit in $\NN(\g)$ has a representative $e\in\g(\mu,2)$ such that the orbit $(\Ad P(\mu))\,e$ is open in $\g(\mu,\ge 2)$, thereby forcing $e\in\V(\mu,2)_{ss}$. Thanks to Lemma~\ref{L5} we also know that $\mu$ is $G$-conjugate to $\tau=\tau_\Delta$. This yields $e\in \H(\Delta)$. 

In view of [\cite{LS}, Table~14.1] we may assume that $e$ is $(\Ad G)$-conjugate to one of the elements $e(i)$ with $i\in\{1,2,3\}$, where
\begin{eqnarray*}
	e(1)&=&e_{234}+e_{1121}+e_{1220}+e_{0122}\  \ \mbox{ in type $\mathrm{(\tilde{A}_2A_1)_2}$},\\
	e(2)&=&e_{123}+e_{0122}+e_{0120}+e_{1222}\ \ \,\mbox{ in type $\mathrm{(C_3(a_1))_2}$},\\
e(3)&=&e_{123}+e_{0120}+e_{4}+e_{1222}\ \ \ \ \ \ \mbox{ in type $\mathrm{(C_3)_2}$}.
\end{eqnarray*}
By Lemma~\ref{L5}, there is a unique $\tau_i\in X_*(T)$ conjugate to $\tau$ and such that $e(i)\in\g(\tau_i,2)$ and $G_{e(i)}\subset P(\tau_i)$. Direct computations show that
$\tau_1=(2,2,-2,2)$, $\tau_2=(2,-2,2,0)$ and $\tau_3=(6,-10,6,2)$. Using [\cite{Bou}, Planche~VIII] one checks directly that in the last two cases $G_{e(i)}$ contains $x_{\alpha_2}(t)$ for every $t\in\k$. 
Then the simple root vector $e_2$ lies in $\Lie(G_{e})\cap \g(\tau, <0)$. As noted in [\cite{LS}, p.~215] we have $x_{-\alpha_1}(t)x_{\alpha_3}(t)\in G_{e(1)}$ for all $t\in \k$ which yields
$\Lie(G_{e(1)})\cap \g(\tau,-2)\ne 0$.  So none of the three cases can occur in our situation.   

Finally, suppose $(\Sigma, p)=({\rm G_2}, 3)$. Thanks to [\cite{LS}, Proposition~13.5]  we only need to consider the  orbit 
$\OO(({\rm\tilde{A}_1)_3})$ which has a nice representative $e'=e_{21}+e_{32}$; see [\cite{LS}, Table~14.1]. 
As $e$ is distinguished and $e'\in \g(\mu, 2)$, where $\mu=(2-2)\in X_*(T)$, it follows from Lemma~\ref{L5} that $\mu$ is $G$-conjugate to $\tau$ and $G_{e'}\subset P(\mu)$.
But  $x_{\alpha_2}(t)\in G_{e'}$ for all $t\in\k$, forcing 
$\Lie(G_{e'})\cap \g(\mu,-2)\ne 0$. This contradiction shows that this case cannot occur either.

Summarising, we have proved that statement~(\ref{S1}) holds for all simple algebraic groups of exceptional types over algebraically closed fields of characteristic $p\ge 0$.
This means that ${\rm LX}(\Delta)=\H(\Delta)$ for all $\Delta\in\mathfrak{D}_G$.
Since $\NN(\g)\subset \Lie(\mathcal{D}G)$ and $Z(G)$ acts trivially on $\g$,
proving Theorem~\ref{main} reduces quickly to the case where $G$ is a simple algebraic group.
For $G$ exceptional, the theorem is a direct consequence of statement~(\ref{S1}). For $G$ classical, the result is known from [\cite{L2}, Theorem~A.2]. The groups of type ${\rm G_2}$, ${\rm F_4}$ and ${\rm E_6}$ were treated earlier in [\cite{Vog}].

  \subsection{}\label{4.13} We would like to finish this paper by a brief discussion of the unipotent analogues of the Lusztig--Xue pieces, ${\rm LX}_u(\Delta)$,  introduced by Lusztig in 
  [\cite{L2}, 2.3].
  
  Let $\mathfrak{U}(G)$ denote the unipotent variety of $G$, the set of all $(\Ad G)$-unstable elements of $G$. The Hesselink stratification 
  $$\mathfrak{U}(G)\,=\,\textstyle{\bigsqcup}_{\Delta\in \mathfrak{D}_G}\,\H_u(\Delta)$$ is described in [\cite{CP}] as follows. Let $\tau=\tau_\Delta$ and write $P(\tau)=Z_{G}(\tau)\,U(\tau)$ where $U(\tau)=R_u(P(\tau))$. Given $k\in\Z_{>0}$ we denote by $U_{\ge k}(\tau)$ the connected normal subgroup of $U(\tau)$ generated by all
  $x_\gamma(t)$ with $t\in\k$ and all $\gamma\in \Sigma$ such that 
  $\la\gamma,\tau\ra\ge k$. It is well known that the factor-group $U_{\ge 2}(\tau)/U_{\ge 3}(\tau)$ is endowed with a natural 
  vector space structure over $\k$ and $\Ad Z_G(\tau)$ acts $\k$-linearly on $U_{\ge 2}(\tau)/U_{\ge 3}(\tau)$. Furthermore, $U_{\ge 2}(\tau)/U_{\ge 3}(\tau)\cong \g(\tau,2)$ as $(\Ad Z_G(\tau))$-modules, and there is a module isomorphism $\bar{\pi}_\Delta\colon\, U_{\ge 2}(\tau)/U_{\ge 3}(\tau)\stackrel{\sim}\longrightarrow \g(\tau,2)$ sending a coset  $\prod_{i=1}^r x_{\beta_i}(t_i)U_3(\tau)$ with $\la \beta_i,\tau\ra=2$ to $\sum_{i=1}^r t_ie_{\beta_i}$, where $e_{\beta_i}$ are root vectors 
  independent of the choice of $(t_1,\ldots,t_r)\in\k^r$; see [\cite{L2}, 2.2], [\cite{LS}, Lemma~18.1] or [\cite{CP}, 3.6]. Composing $\bar{\pi}_\Delta$ with the canonical homomorphism $U_{\ge 2}(\tau)\to U_{\ge 2}(\tau)/U_{\ge 3}(\tau)$ we obtain a natural  surjection 
  $\pi_\Delta\colon\, U_{\ge 2}(\tau)\twoheadrightarrow\, \g(\tau,2)$.
  
  It follows from [\cite{CP}, Theorems~3.6 and 5.2]  that $$\H_u(\Delta)\,=\,(\Ad G)\,\big(\pi_\Delta^{-1}(\V(\tau_\Delta,2)_{ss})\big)$$ for every $\Delta\in\mathfrak{D}_G$. In [\cite{L2}, 2.3], the unipotent pieces ${\rm LX}_u(\Delta)$ 
  are defined in a similar fashion except that $\V(\tau_\Delta,2)_{ss}$ is replaced by an {\it a priory} larger set $\g_2^{\Delta,!}$. More precisely, $${\rm LX}_u(\Delta)\,=\,(\Ad G)\,\big(\pi_\Delta^{-1}(\g_2^{\Delta,!})\big).$$
  
  In [\cite{L2}, Theorem~2.4], Lusztig proved that 
   \begin{equation}\label{EU}
   	\mathfrak{U}(G)\,=\,\textstyle{\bigsqcup}_{\Delta\in \mathfrak{D}_G}\,
  {\rm LX}_u(\Delta)\end{equation}
   when $G$ is a simple algebraic group of type ${\rm A}$, $\rm B$, $\rm C$ or $\rm D$, and he expected that (\ref{EU}) would continue to hold for all connected reductive groups. 
 Our next result shows that this expectation was correct.
 
 \begin{corollary}
  	Let $G$ be a connected reductive group over an algebraically closed field. Then 
  	${\rm LX}_u(\Delta)=\H_u(\Delta)$ for all $\Delta\in\mathfrak {D}_G$ and 
  	(\ref{EU}) holds for $\mathfrak{U}(G)$.
\end{corollary}
  	\begin{proof}
  		Theorem~\ref{main} in conjunction with
  	 the preceding discussion shows that ${\rm LX}_u(\Delta)=\H_u(\Delta)$ for all $\Delta\in\mathfrak{D}_G$. Since the Hesselink strata $\H_u(\Delta)$ with $\Delta\in\mathfrak{D}_G$ form a partition of $\mathfrak{U}(G)$ by [\cite{CP}, Theorem~5.2], we deduce that (\ref{EU}) holds for $\mathfrak{U}(G)$.
  	 \end{proof}

\end{document}